\documentclass[11pt]{amsart}
\usepackage[margin=1.25in]{geometry}
\usepackage[utf8]{inputenc}
\usepackage[usenames, dvipsnames]{color}
\usepackage{tikz-cd}
\usepackage{bm}
\usepackage{amssymb}
\input xy
\usepackage[all]{xy}
\usepackage{tikz-cd}
\usepackage{placeins}
\usepackage{lscape}
\usepackage{dsfont}

\newcommand{\unit}{\text{\textbf{1}}}
\newcommand{\Hc}{{\mathcal H}}

\newcommand{\Inv}{\operatorname{Inv}}
\newcommand{\sVec}{\operatorname{sVec}}
\newcommand{\Sem}{\operatorname{Sem}}

\newcommand{\K}{\mathbb{C}}

\newcommand{\hB}{\widehat{B}}
\renewcommand{\Vec}{\text{Vec}}

\newcommand{\cC}{{\mathcal C}}
\newcommand{\cF}{{\mathcal F}}
\newcommand{\cA}{{\mathcal A}}
\newcommand{\cZ}{{\mathcal Z}}
\newcommand{\cT}{{\mathcal T}}
\newcommand{\cE}{{\mathcal E}}

\newcommand{\Z}{{\mathbb Z}}

\newcommand{\id}{\operatorname{id}}

\newcommand{\cB}{\mathcal{B}}
\newcommand{\cP}{\mathcal{P}}
\newcommand{\cD}{\mathcal{D}}

\newcommand{\Gbox}{\stackrel{G}{\boxtimes}}
\newcommand{\CGboxD}{\cC\Gbox\cD}

\newcommand{\Fun}{\operatorname{Fun}}

\newcommand{\Rep}{\operatorname{Rep}}

\newcommand{\Id}{\operatorname{Id}}

\newcommand{\Aut}{\operatorname{Aut}}

\newcommand{\TGB}{\cT(G,\cB_z)}
\newcommand{\ndTBB}{\cT^{nd}(B,\cB_z)}

\theoremstyle{definition}

\numberwithin{equation}{section}

\newtheorem{theorem}{Theorem}[section]
\newtheorem{lemma}[theorem]{Lemma}

\newtheorem{corollary}[theorem]{Corollary}

\theoremstyle{definition}
\newtheorem{definition}[theorem]{Definition}
\newtheorem{example}[theorem]{Example}

\theoremstyle{remark}

\newtheorem{remark}[theorem]{Remark}

\theoremstyle{remark}

\author[C. Delaney]{Colleen Delaney}
\address{Department of Mathematics, Purdue University}
\email{colleend@purdue.edu}
\author[C. Galindo]{C\'esar Galindo}
\address{ Departamento de Matem\'aticas, Universidad de los Andes, Bogot\'a, Colombia}
\email{cn.galindo1116@uniandes.edu.co}

\author[J. Plavnik]{Julia Plavnik}
\address{Department of Mathematics, Indiana University, Bloomington and }
\address{Department of Mathematics and Data Science,  Vrije Universiteit Brussel, Belgium}
\email{jplavnik@iu.edu}
\author[E. Rowell]{Eric C. Rowell}
\address{Department of Mathematics, Texas A\&M University}
\email{rowell@math.tamu.edu}
\author[Q. Zhang]{Qing Zhang}
\address{Department of Mathematics, University of California, Santa Barbara}
\email{qingzhang@ucsb.edu}
\begin{document}

\title{The Condensed Fiber Product and Zesting}

\thanks{The authors gratefully acknowledge the support of the American Institute of Mathematics, where this collaboration was initiated. This material is also based upon work supported by the National Science Foundation under Grant No. DMS-1928930, while a subset of the authors were in residence at SLMath in Berkeley, California, during the summer of 2024. We thank D. Nikshych, V. Ostrik, D. Penneys, X.-G. Wen, S.-H. Ng and D. Reutter for useful conversations. C.G. was partially supported by Grant INV-2023-162-2830 from the
School of Science of Universidad de los Andes.
C.D. was partially supported by Simons Foundation Grant \#888984 (Simons Collaboration on Global Categorical Symmetries). J.P was partially supported by US NSF Grant {DMS-2146392} and by Simons Foundation Award \#889000 as part of
the Simons Collaboration on Global Categorical Symmetries. E.C.R. was partially supported by US NSF Grant {DMS-2205962}} 

\begin{abstract}
We introduce the condensed fiber product of two $G$-crossed braided fusion ca\-tegories, generalizing existing constructions in the literature. We show that this product is closely related to the cohomological construction known as zesting. Furthermore, the condensed fiber product defines a monoidal structure on the 2-category of $G$-crossed braided extensions of braided categories with a fixed transparent symmetric pointed subcategory, and the same holds for a certain subcategory of non-degenerate extensions.
\end{abstract}



\maketitle


\section{Introduction}

The zesting construction \cite{DGPRZ,DGPRZ2} has been of significant utility in both physics and mathematics: it has yielded realizations of previously unknown theories \cite{AIM2012,Palcouxetal,BRW:realizing_modulardate}, provided insights into theories with indistinguishable modular data \cite{isotopes}, and illuminated symmetries in condensed matter systems \cite{ABK}. One feature of zesting is its explicit and intrinsic cohomological formulation. However, it is desirable to have an extrinsic, physically intuitive description. The condensed fiber product, defined below, provides such a description in many cases. Briefly, the condensed fiber product can be viewed as a certain condensation of two stacked theories.

Fix a finite abelian group \( B \) with group of characters \( \hB \), and an element \( z \in B \) such that \( z^2 = 1 \).
Any symmetric pointed fusion category is determined by a pre-metric group \( \mathcal{B}_z = \mathcal{C}(\hB, q_z) \) with quadratic form defined by \( q_z(\varphi) = \varphi(z) \), for some choice of \( B \) and \( z \) (see \cite[Section 8.4]{EGNObook}).  
Now let \( \mathcal{C} \) be a non-degenerate braided fusion category containing \( \mathcal{B}_z \), and suppose \( \mathcal{D} \) is another such non-degenerate category. We have the following inclusions:
\[
\operatorname{Rep}(B) \cong \nabla(\mathcal{B}_z) := \langle (\varphi, \varphi^{-1}) \rangle \subset \mathcal{B}_z \boxtimes \mathcal{B}_z \subset \mathcal{C} \boxtimes \mathcal{D},
\]
where \( \nabla(\mathcal{B}_z) \cong \operatorname{Rep}(B) \) is Tannakian, i.e., bosonic. Since the Deligne product (stacked) category \( \mathcal{C} \boxtimes \mathcal{D} \) is non-degenerate and \( \nabla(\mathcal{B}_z) \) can be condensed, we obtain a new non-degenerate category \( [(\mathcal{C} \boxtimes \mathcal{D})_B]_e \), i.e., the even sector after condensation. In the physics literature, this is known as \emph{boson condensation} \cite{SlingerlandBais}, while in the mathematics literature, it is described as the category of local \( \hB \)-modules in \( \mathcal{C} \boxtimes \mathcal{D} \) \cite[Definition 3.12]{DMNO}. We note that \( [(\mathcal{C} \boxtimes \mathcal{D})_B]_e \) also contains \( (\mathcal{B}_z \boxtimes \mathcal{B}_z)_B \cong \mathcal{B}_z \), and thus is, in some sense, in the same class as \( \mathcal{C} \) and \( \mathcal{D} \). Special cases and closely related constructions are found in both the physics \cite{HsinLamSeiberg,LanWen,LanKongWen,CarolynZhang} and mathematics \cite{DGNOI,DNO,16foldway,Nik2022} literature.

In particular, the condensed fiber product as it applies to topological order shares some structural similarities with the {\it hierarchy construction}, which organizes families of fractional quantum Hall states \cite{LanWen} and was recently given a categorical description in \cite{CarolynZhang}. Performing the hierarchy construction on a $U(1)$-enriched supermodular tensor category $\mathcal{C}$ corresponding to fractional quantum Hall states is a two-step process that (1) takes the Deligne product $\mathcal{C} \boxtimes \mathcal{D}$ with $\mathcal{D}$ suitably chosen and the (2) condensation of a certain algebra of bosons in $\mathcal{C} \boxtimes \mathcal{D}$.

This two-step process of stacking and condensing appears in both settings and seems to be responsible for producing some of the same examples: the topological orders in Kitaev's 16-fold way can be produced via the condensed fiber product (in this case zesting) as well the hierarchy construction. However, the constructions generalize one another is somewhat orthogonal directions and each is capable of producing examples that is beyond the scope of the other. On one hand, zesting is realized by means of the condensed fiber product of two (for us bosonic) topological orders $\mathcal{C}$ and $\mathcal{D}$ that share a local symmetry $\Rep(G,z)$ and $\mathcal{D}$ is pointed with minimal dimension and results in a topological order with the same rank as $\mathcal{C}$. On the other hand, the selection of an appropriate $\mathcal{D}$ in performing the hierarchy construction on $\mathcal{C}$ with respect to $\mathcal{D}$ requires some properties to be satisfied by their filling fractions and $U(1)$-symmetry fractionalization patterns. It would be worthwhile to clarify when these two constructions agree since this would suggest an interpretation of the condensed fiber product as a kind of generalized condensation as suggested in \cite{CarolynZhang}. \\

There are two natural generalizations of the version of the condensed fiber product construction presented above, which are similarly closely related to various existing notions in tensor category theory and its applications to condensed matter theory:  

First, suppose that \( \mathcal{C} \) and \( \mathcal{D} \) are non-degenerate braided fusion categories both containing a (not necessarily pointed) symmetric category \( \mathcal{F} \cong \operatorname{Rep}(G, z) \). Then the work of \cite{DNO} shows that \( \mathcal{C} \boxtimes \mathcal{D} \) contains a condensible algebra object \( A \cong \operatorname{Fun}(G) \), and the category of local \( A \)-modules, \( (\mathcal{C} \boxtimes \mathcal{D})_A^{\text{loc}} \), is again non-degenerate and contains \( \mathcal{F} \) as a subcategory. This construction is considered in detail in \cite[Section 4]{LanKongWen}.

Second, observe that since \( \mathcal{B}_z \subset \mathcal{C} \), we have a canonical \( B \)-grading of \( \mathcal{C} \) given by
\[
\mathcal{C}_b := \langle X : c_{X, \varphi} c_{\varphi, X} = \varphi(b) \operatorname{Id}_{\varphi \otimes X} \;\text{for all}\; \varphi \in \hB \rangle.
\]  Moreover, since \( \mathcal{B}_z \) is symmetric, we see that \( \mathcal{C}_e \cap (\mathcal{C}_e)^\prime \supset \mathcal{B}_z \), and the non-degeneracy of \( \mathcal{C} \) implies that the grading is faithful. Finally, one computes that the centralizer \( \nabla(\mathcal{B}_z)^\prime \subset \mathcal{C} \boxtimes \mathcal{D} \) of \( \nabla(\mathcal{B}_z) \) is the so-called \emph{fiber product} \cite{Nik2018},
\[
\mathcal{C} \stackrel{B}{\boxtimes} \mathcal{D} = \bigoplus_{b \in B} \mathcal{C}_b \boxtimes \mathcal{D}_b,
\]
and thus \( [(\mathcal{C} \boxtimes \mathcal{D})_B]_e \cong (\mathcal{C} \stackrel{B}{\boxtimes} \mathcal{D})_B \). Therefore, we can apply the above construction more generally to \( G \)-crossed braided fusion categories \( \mathcal{C} \) and \( \mathcal{D} \) such that \( \mathcal{B}_z \subset \mathcal{C}_e \cap (\mathcal{C}_e)^\prime \), and similarly for \( \mathcal{D} \). That is, we construct the \( G \)-graded category \( (\mathcal{C} \stackrel{G}{\boxtimes} \mathcal{D})_B \), which we call the \emph{condensed fiber product}. This is the main topic of this article. Remarkably, we find that \( (\mathcal{C} \stackrel{G}{\boxtimes} \mathcal{D})_B \) is again a \( G \)-crossed braided fusion category containing \( \mathcal{B}_z \), hinting that this class of categories forms a monoidal 2-category.

The paper is organized as follows. In Section~\ref{sec:preliminaries}, we provide the necessary preliminaries on braided fusion categories, $G$-crossed braided fusion categories, and the zesting construction. In Section~\ref{sec:categories}, we introduce the categories $\cT(G, \cB_z)$ (and $\cT^{nd}(B, \cB_z)$) of $G$-crossed extensions of braided categories with a transparent symmetric pointed category $\cB_z$ and discuss their basic properties, along with detailed examples illustrating their construction. We also introduce the condensed fiber product and show that it endows $\TGB$ and $\ndTBB$ with a monoidal structure. In Section~\ref{sec:zesting_cfp}, we explore the relationship between zesting and the condensed fiber product, demonstrating that the condensed fiber product can be realized as a zesting in the case of pointed fusion extensions. Finally, in Section~\ref{sec:discussion}, we discuss open problems and potential directions for future research, such as extensions of the condensed fiber product to more general settings and its relation to generalized zesting.

\section{Preliminaries}\label{sec:preliminaries}
We refer to \cite{EGNObook} for the definitions of braided/symmetric tensor, fusion, and (bi-)module categories, which we use freely in this paper. A fusion category is \emph{pointed} if every simple object is $\otimes$-invertible. The group of isomorphism classes of invertible objects in a fusion category $\cC$ is denoted by $\Inv(\cC)$. To fix notation, we briefly review the cohomological description of pointed (braided) fusion categories associated with finite (abelian) groups and group cohomology (e.g., \cite{EGNObook,JS93}).

For a finite group $G$ and a $3$-cocycle $\omega\in Z^3(G,\K^\times)$, we denote by $\cC(G,\omega)$ the fusion category of $G$-graded vector spaces with simple objects $\K_g$ for $g\in G$, where the tensor product has an associator given by $\omega$.

For a finite abelian group $B$ and an abelian $3$-cocycle $(c,\omega)\in Z_{\mathrm{ab}}^3(B,\K^\times)$, the category $\cC(B,\omega)$ has a braiding defined by $c$. The map $q_c: B \to \K^\times$, given by $b \mapsto c(b,b)$, is a quadratic form, meaning $q(b^{-1}) = q(b)$, and the map $q(a+b)q(a)^{-1}q(b)^{-1}$ is bimultiplicative. The quadratic form $q_c$ determines the cohomology class of the associated abelian $3$-cocycle. Thus, we denote by $\cC(B,q)$ a pointed braided fusion category with $\Inv(\cC) = B$ and quadratic form $q$.

When $q$ is a group homomorphism, the category $\cC(B)$ with braiding given by the bicharacter
\[
c_q(a,b) = 
\begin{cases}
-1 & \text{if } q(a) = q(b) = -1, \\
1 & \text{otherwise},
\end{cases}
\]
is a symmetric fusion category, denoted by $\cC(B,q)$. Conversely, every pointed symmetric fusion category has this form.

The categories $\cC$ we study will be (faithfully) graded by a finite group, that is, $\cC=\bigoplus_{g\in G} \cC_g$ as an abelian category with each $\cC_g$ non-trivial and for $X_g\in\cC_g,Y_h\in\cC_h$ we have $X_g\otimes Y_h\in \cC_{gh}$.

\begin{definition}
\label{def:gcrossedbraided}
Let $G$ be a finite group. A (right) $G$-crossed braided category is a fusion category $\cC$ with
\begin{enumerate}
\item  a $G$-grading $\cC = \oplus_{g \in G}\, \cC_g$, 
\item a {\it right} $G$-action,  that is, a monoidal functor $T: \underline{G^{op}} \to \underline{\Aut_{\otimes}(\cC)}$, denoted $T(g)=g^*$, such that $g^*(\cC_g) \subset \cC_{h^{-1}gh}$ for all $g,h\in G$,  
\item a $G$-braiding consisting of natural isomorphisms $c_{X,Y_h}: X \otimes Y_h \to Y_h \otimes h^*(X)$ for all $X\in \cC$, $Y_h\in \cC_h$, satisfying the right action analogues of the axioms in \cite[Definition 8.24.1 (a)-(c)]{EGNObook}. 
\end{enumerate}     
\end{definition}

We will use the notation $X^g:=g^*(X)$ for the (right) $G$-action on objects and $f^g:=g^*(f)$ for the $G$-action on morphisms.

The following product on $G$-graded categories will be important:
\begin{definition}[\cite{BBCW,Nik2018}] 
Let $\cC$ and $\cD$ be $G$-graded categories.  The \emph{fiber product} of $\cC$ and $\cD$ is the $G$-graded fusion category
$\CGboxD:=\bigoplus_{g\in G} \cC_g\boxtimes \cD_g$.
\end{definition}  Here $\boxtimes$ is the Deligne product of categories. 
It is clear that $\CGboxD$ is $G$-graded, via the diagonal embedding  $G\subset G\times G$, and it is not hard to see that $\CGboxD$ has a canonical structure as a $G$-crossed braided fusion category (cf. \cite[Section 6]{GreenNik}).

Next we recall the zesting construction:
\begin{definition}[Definition 3.1, Remark 3.9 in \cite{DGPRZ2}]\label{def:gxzesting}
Let $\cC$ be a $G$-crossed braided fusion category. A $G$-crossed zesting datum $(\lambda,\nu)$ on $\cC$ consists of a map $\lambda: G\times G\to \operatorname{Ob}(\cC_e)$, where each $\lambda(g_1,g_2)$ is an invertible object with $\lambda(g_1,e)=\lambda(e,g_1)=\unit$, and isomorphisms

\begin{equation}\label{eq: 2-cocycle condition}
\nu_{g_1,g_2,g_3}: \lambda(g_1,g_2)^{g_3} \otimes \lambda(g_1g_2,g_3) \to \lambda(g_2,g_3) \otimes \lambda(g_1,g_2g_3),
\end{equation}
satisfying $\nu_{g_1,e,g_2}=\id$ and the {\it G-crossed associative zesting condition}

    \begin{align}\label{eq:zestingcond}
    \begin{split}
    \left ( \nu_{g_2,g_3,g_4} \otimes \id_{\lambda(g_1,g_2g_3g_4)} \right) \circ \left ( \id_{\lambda(g_2,g_3)^{g_4}} \otimes \nu_{g_1,g_2g_3,g_4}\right ) \circ \left ( (\nu_{g_1,g_2,g_3})^{g_4} \otimes \id_{\lambda(g_1g_2g_3,g_4)})\right ) \\= \left ( \id_{\lambda(g_3,g_4)} \otimes \nu_{g_1,g_2,g_3g_4} \right ) \circ \left ( c_{\lambda(g_1,g_2)^{g_3g_4},\lambda(g_3,g_4)} \otimes \id_{\lambda(g_1g_2,g_3g_4)}\right) \circ \left ( \id_{\lambda(g_1,g_2)^{g_3g_4} \otimes \nu_{g_1g_2,g_3,g_4}} \right ),
    \end{split}
    \end{align}
    for all $g_1,g_2,g_3,g_4 \in G$.
\end{definition}

\begin{remark}
A $G$-crossed zesting datum $(\lambda,\nu)$ induces a normalized map $\overline{\lambda} : G\times G\to \Inv(\cC_e)$. The isomorphism \eqref{eq: 2-cocycle condition} shows that $\overline{\lambda}$ satisfies 
\[
\overline{\lambda}(g_1,g_2)^{g_3} \cdot \overline{\lambda}(g_1g_2,g_3) = \overline{\lambda}(g_2,g_3) \cdot \overline{\lambda}(g_1,g_2g_3),
\]
hence $\overline{\lambda} \in Z^2(G,\Inv(\cC_e))$, where $\Inv(\cC_e)$ is a $G$-module via the induced $G$-action on $\cC$. By abuse of notation, we will refer to $\lambda$ as a 2-cocycle, even though the actual 2-cocycle is the induced map on $\Inv(\cC_e)$.
\end{remark}

A $G$-crossed zesting datum generalizes an \emph{associative} zesting datum for a $G$-graded \emph{braided} fusion category $\cC$ (where $G$ is abelian), as defined in \cite{DGPRZ}. The associative zesting is recovered by viewing a $G$-graded braided fusion category as a braided $G$-crossed category with trivial $G$-action. Additionally, for $G$-graded fusion categories, there is a refinement called \emph{braided} $G$-zesting, which modifies the braiding on $\cC$ using an isomorphism $t(a,b): \lambda(a,b) \rightarrow \lambda(b,a)$ and a function $j:G\rightarrow \Aut_\otimes(\Id_\cC)$ (see \cite[Definition 4.1]{DGPRZ}). When the 2-cocycle $\lambda$ takes values in a \emph{symmetric} pointed subcategory of $\cC$, a braided zesting datum consists of a $G$-crossed braided zesting datum $(\lambda,\nu)$ along with a \emph{trivialization} of the zested $G$-action functor $T^{\lambda,\nu}$ \cite[Theorem 3.16]{DGPRZ2}.

\section{The categories $\cT(G,\cB_z)$ and $\cT^{nd}(B,\cB_z)$}\label{sec:categories}

Fix a finite group $G$, a finite abelian group $B$, and an element $z \in B$ such that $z^2 = 1$. Let $\hB$ be the abelian group of linear characters on $B$ and define a quadratic form (which is also a character) on $\hB$ by $q_z(\varphi) := \varphi(z) \in \{\pm 1\}$. Let $\cB_z := \cC(\hB,q_z)$ be the corresponding symmetric pointed braided fusion category.

Denote by $\cT(G,\cB_z)$ the category of braided $G$-crossed fusion categories $\cC$ equipped with a fully faithful braided embedding $\iota: \cB_z \to \cC_e$, such that $\cB_z \subseteq \cC_e \cap (\cC_e)'$, meaning $c_{Y,X} \circ c_{X,Y} = \id_{X \otimes Y}$ for all $X \in \cB_z$ and $Y \in \cC_e$.

\begin{definition}\label{def:equivalence braided GX}
Following \cite[Definition 8.10]{DN}, for $\cC, \cD \in \cT(B, \cB_z)$, an equivalence $F: \cC \to \cD$ of $B$-extensions $\iota_\cC,\iota_\cD$, is called an \emph{equivalence of $B$-crossed extensions of $\cB_z$} if $\iota_\cD = \operatorname{Ind}(F) \circ \iota_\cC$. Here, $\operatorname{Ind}(F): \mathcal{Z}(\cC) \to \mathcal{Z}(\cD)$ denotes the braided equivalence induced by $F$. For an alternative but equivalent definition of equivalence of braided $B$-crossed categories, see \cite{Galindo2017}.
\end{definition}

We are particularly interested in the following special case: suppose $\cC$ is a modular category or, more generally, a \emph{non-degenerate} braided fusion category, and $\cB_z \subset \cC$ appears as a braided fusion subcategory. Then $\cC$ admits a canonical faithful $B$-grading given by \[\cC_b:=\langle X: c_{X,\varphi}\circ c_{\varphi,X}=\varphi(b)\id_{\varphi\otimes X} \;\text{for all}\; \varphi\in\hB\rangle.\]
Viewing $\cC$ as a braided $B$-crossed fusion category with trivial $B$-action we can see $\cC \in \cT(B,\cB_z)$.
We denote by $\cT^{nd}(B,\cB_z)$ the subcategory of non-degenerate braided fusion categories in $\cT(B,\cB_z)$, where the $B$-grading is determined by the inclusion $\cB_z \subset \cC$. The objects of $\cT^{nd}(B,\cB_z)$ correspond precisely to the \emph{minimal modular} extensions of $\cB_z$, as defined in \cite{MugerPLMS} and \cite{LanKongWen} for a general braided fusion category.

\begin{example} The following are canonical examples of categories in $\TGB$ and $\ndTBB$.
\begin{itemize}
    \item  Consider the pointed fusion category $\Vec_G\boxtimes \cB_z$ with $G$-action $(g,\varphi)^h\mapsto (h^{-1}gh,\varphi)$ and the natural $G$-braiding: $c_{(g,\varphi),(h,\psi)}=\Id_{gh}\times c_{\varphi,\psi}$, where $c_{\varphi,\psi}$ is the braiding on $\cB_z$.  With this $G$-crossed structure $\Vec_G\boxtimes\cB_z\in\TGB$.
    \item Notice that the center $\cZ(\cB_z)$ of $\cB_z$ is a member of $\ndTBB$.  An explicit description of $\cZ(\cB_z)$ from a metric group is $\cC(\hB\times B,Q_z)$ with $Q_z(\varphi,b)=\varphi(bz)$ 
 see \cite[Proposition 5.8]{MR3107567}.  It is clear that $\cB_z\cong \langle (\varphi,e)\rangle$, yielding a $B$-crossed braided structure. Notice that, as a fusion category $\cZ(\cB_z)\equiv \Vec_B\boxtimes\cB_z$, so that on the level of fusion categories this example is equivalent to the previous example when $G=B$.
\end{itemize}
   
\end{example}

The members of $\cT(G,\cB_z)$ have a number of special properties, one of which is the following: 
\begin{theorem}\label{thm: trivial action}
For $\cC \in \cT(G, \cB_z)$, the functors $g^*$ of the $G$-action, when restricted to $\cB_z$, are tensor equivalent to the identity.
\end{theorem}
Theorem \ref{thm: trivial action} follows from the more general lemma below, given that $\cB_z \subset \cC_e \cap (\cC_e)^\prime$ for any $\cC \in \cT(G, \cB_z)$. In particular, when considering a $G$-crossed braided zesting of any $\cC \in \cT(G, \cB_z)$ with a $2$-cocycle $\lambda$ taking values in $\cB_z$, we can omit the superscripts in equation \eqref{eq: 2-cocycle condition} of Definition \ref{def:gxzesting}.

\begin{lemma}\label{lem:trivial action}
Let $\mathcal{B}$ be a $G$-crossed braided category and $\cD\subset \mathcal{Z}_2(\mathcal{B}_e)$ a transparent fusion subcategory. Then the functors $g^*$ of the $G$-action, when restricted to $\cD$, are tensor equivalent to the identity.
\end{lemma}

\begin{proof}
From \cite[Theorem 7.12]{ENO3}, we know that a braided $G$-crossed extension $\mathcal{B}$ can be defined via a categorical 2-group homomorphism $BG \to B Pic(\mathcal{B}_e)$. The $G$-action on $\mathcal{B}_e$ arises from the induced action of the categorical group homomorphism via $\alpha$-induction:
\[
\partial : B Pic(\mathcal{B}_e) \to B\Aut_{\otimes} (\mathcal{B}_e).
\]

Consider the invertible $\mathcal{B}_e$-module $\mathcal{B}_g$. The $\alpha$-induction for $\mathcal{B}_g$ is characterized by the tensor equivalences:
\begin{align*}
    \alpha^{\pm}_{\mathcal{B}_g}:\mathcal{B}_e^{\mathrm{op}} \to \mathrm{End}_{\mathcal{B}_e}(\mathcal{B}_g), && \alpha^{\pm}(X)(V) = X \otimes V, && X \in \mathcal{B}_e, V \in \mathcal{B}_g,
\end{align*}
where the $\mathcal{B}_e$-module structures on $\mathcal{B}_g$, used in defining $\alpha^{\pm}$, are given by:
\begin{align*}
\alpha^+&:=c_{X,Y} \otimes \mathrm{id}_V: \alpha^{+}(X)(Y\otimes V) \to Y\otimes \alpha^{+}(X)(V), \\
\alpha^-&:=c_{Y,X}^{-1} \otimes \mathrm{id}_V : \alpha^{-}(X)(Y\otimes V) \to Y\otimes \alpha^{-}(X)(V),
\end{align*}
assuming for simplicity that $\mathcal{B}$ is strict. 

The braided tensor equivalence $g^*:\mathcal{B}_e \to \mathcal{B}_e, X\mapsto g^*(X)$ is determined by 
\begin{equation}\label{eq: condition alpha-induction an action}
\alpha^{+}(X) \cong \alpha^{-}(g^*(X)),   
\end{equation}
considering $\alpha^{\pm}$ as $\mathcal{B}_e$-module functors, see \cite[Proposition 8.11]{DN} for more details.

Now, if $\mathcal{D} \subseteq \mathcal{B}_e$ is transparent, meaning that
\[
c_{X,Y} = c_{Y,X}^{-1} \quad \text{for all} \quad Y \in \mathcal{B}_e, \, X \in \mathcal{D},
\]
then $\alpha^{+}(X) = \alpha^{-}(X)$ for any $X\in \mathcal{D}$. Therefore, equation \eqref{eq: condition alpha-induction an action} implies that $g^*|_{\cD}\cong \Id_{\cD}$ as tensor functors.

\end{proof}

\subsection{$G$-crossed and braided zestings of $\cT(G,\cB_z)$ and $\cT^{nd}(B,\cB_z)$}
Next we find that the categories in $\cT(G,\cB_z)$ and $\cT^{nd}(B,\cB_z)$ have the same sets of $G$-crossed zestings, respectively, braided $B$-zestings:

\begin{lemma}\label{lem:one datum to zest them all}
\begin{enumerate}
    \item[(a)]
    Suppose that $\cC,\cD\in\cT(G,\cB_z)$.  Then the collection of  $G$-crossed zesting data $(\lambda,\nu)$ for $\cC$ with $\lambda:G\times G\rightarrow \cB_z$ is identical to that of $\cD$.
    \item[(b)] Suppose that $\cC,\cD\in\cT^{nd}(B,\cB_z)$.  Then the collection of braided $B$-zesting data $(\lambda,\nu,t)$ \emph{with $j=\id$} and $\lambda:B\times B\rightarrow \cB_z$ for $\cC$ is identical to that of $\cD$. 
    \end{enumerate}
\end{lemma}
\begin{proof}
Statement (a) is just the observation that the definition of a $G$-crossed zesting datum only depends on the group $G$, the target subcategory $\cB_z$, and the action of $G$ on $\cB_z$.

Statement (b) requires some justification since the braided zesting conditions in general depend on the specific categories.  We must verify conditions (BZ1) and (BZ2) of  \cite[Definition 4.1]{DGPRZ}. Condition (BZ1) is
\[\omega(a,b):=\chi_{\lambda(a,b)}\circ j_{ab}\circ j_a^{-1}\circ j_b^{-1}\in\Aut_\otimes^B(\Id_\cC)\] where $\chi_{\varphi}(X)$ is defined via $c_{X,\varphi}\circ c_{\varphi,X} =\chi_\varphi(X)\otimes \id_\varphi$.
We claim that this is satisfied by $j=\id$: indeed, since $\lambda(a,b)\in\cB_z$ we have that $\chi_{\lambda(a,b)}\in\Aut_\otimes^B(\Id_\cC)$ with no adjustment needed: i.e., $\chi_{\lambda(a,b)}|_{\cC_e}=\Id_{\cC_e}$ using the fact that $\cC_e=(\cB_z)^\prime$. Now the remaining condition (BZ2) for $(\lambda,\nu,t)$ to be a braided $B$-zesting (illustrated in \cite[Figures 7 and 8]{DGPRZ}) depends only on quantities that are constant for all members of $\cT^{nd}(B,\cB_z)$.
\end{proof}
\begin{remark}\begin{itemize}
    \item 

    In the most general situation of \cite{DGPRZ} the automorphism $\chi_{\lambda(a,b)}$ may not be constant on the graded components, particularly if the grading is not the universal one.  
    \item Notice that both $\cT(G,\cB_z)$ and $\cT^{nd}(B,\cB_z)$ contain pointed categories, i.e. $\Vec_G\boxtimes \cB_z$ and $\cZ(\cB_z)$.  Thus one can determine \emph{all} $G$-crossed zesting data on $\TGB$, respectively, braided $B$-zesting data on $\ndTBB$ by means of purely cohomological computations.  
    \end{itemize}
\end{remark}


\begin{lemma}
\label{lem:allNDpointedzest}
A braided $B$-zesting of a minimal modular extension of $\cB_z$ is again a minimal modular extension of $\cB_z$, and all pointed minimal modular extensions of $\cB_z$ can be obtained by a braided $B$-zesting of $\mathcal{Z}(\cB_z)$.
\end{lemma}

\begin{proof}

The set of equivalence classes of pointed minimal modular extensions of $\cB_z$, denoted by $Mext_{pt}(\cB_z)$, has an abelian group structure. In \cite[Equation (58)]{Nik2022}, the following exact sequence was constructed:

\[
H^3_{ab}(B, \K^\times) \xrightarrow{\iota} Mext_{pt}(\cB_z) \xrightarrow{\lambda} H_{ab}^2(B, \hB) \xrightarrow{PW^2} H_{ab}^4(B, \K^\times),
\]
where $\lambda$ corresponds to the cohomology associated with the abelian extension 
\[
0 \to \hB \to \operatorname{Inv}(\cP) \overset{\pi}{\to} B \to 0,
\]
with $\pi$ corresponding to the $B$-grading. 

The map $PW^2$ is the Pontryagin-Whitehead homomorphism defined in \cite{Nik2022}. By the definition of $PW^2$, for a given $\lambda \in Z^2_{ab}(B, \hB)$, the image $PW^2(\lambda)$ has trivial cohomology if there exist $\omega \in C^3(B, \K^\times)$ and $c \in C^2(B, \K^\times)$ such that $(\lambda, \omega, c)$ is a braided zesting datum. In this case, the zesting of $\mathcal{Z}(\cB_z)$ associated with $(\lambda, \omega, c)$ corresponds to the minimal modular extension.

\end{proof}

In fact, the following is immediate from \cite[Proposition 5.12]{DGPRZ}:
\begin{lemma}\label{lem:allBzestNDareND}
 Let $\cC\in \cT^{nd}(B,\cB_z)$. Then $\cC^{(\lambda,\nu,t)}$ is non-degenerate for any braided $B$-zesting datum $(\lambda,\nu,t)$ with $j=\id$.
 \qed
\end{lemma}

Therefore, not every pointed braided category $\cP$ in $\cT(B, \cB_z)$ with $\cP_e\cong \cB_z$ is braided equivalent to a braided $B$-zesting of $\cZ(\cB_z)$. For example, no symmetric category satisfies this equivalence. A complete description of the braided zestings of $\operatorname{Vec}_B \boxtimes \cB_z$ is provided in \cite[Section 8.7]{DN}. Specifically, the braided zestings of $\operatorname{Vec}_B \boxtimes \cB_z$ correspond precisely to the braided quasi-trivial $B$-extensions of $\cB_z$. Since $\cB_z$ is pointed, these quasi-trivial extensions reduce to pointed braided $B$-extensions of $\cB_z$. 

According to \cite[Section 8.7]{DN}, each quasi-trivial braided $B$-extension of $\cB_z$ has an associated group homomorphism $f: B \to \operatorname{Aut}_{\otimes}(\operatorname{Id}_{\cB_z})$, where $\operatorname{Aut}_{\otimes}(\operatorname{Id}_{\cB_z}) \cong \widehat{B}$ is the group of tensor automorphisms of the identity functor on $\cB_z$. Equivalently, this corresponds to a bicharacter $f: B \times B \to \mathbb{C}^*$. Moreover, two braided $B$-extensions with homomorphisms $f_1$ and $f_2$ are related by a braided $B$-zesting if and only if $f_1 = f_2 \cdot \chi_Z$, where $\chi_Z$ is the character corresponding to $Z \in \operatorname{Inv}(\cB_z)$.

This observation contrasts with the case of pointed braided $B$-crossed extensions of $\cB_z$, which, as we will see as a consequence of a special case of Corollary \ref{lem: pointed is trivial homomorphism} in Section \ref{sec:zesting_cfp}, are all related by $B$-crossed zestings.

As an example of the contrast, consider the braided $B$-crossed extensions associated with $\operatorname{Vec}_B \subset \operatorname{Vec}_{B \times \widehat{B}}$, with inclusion map $x \mapsto (x, 0)$. First, we observe that
\[
\operatorname{Vec}_{B \times \widehat{B}} \subset \mathcal{Z}(\operatorname{Vec}_{B \times \widehat{B}}) = \operatorname{Vec}_{B \times \widehat{B} \times \widehat{B} \times B}, \quad (a, \alpha) \mapsto (a, \alpha, 0, 0),
\]
when $\operatorname{Vec}_{B \times \widehat{B}}$ is regarded as a Tannakian category. Alternatively, we consider
\[
\operatorname{Vec}_{B \times \widehat{B}} \subset \mathcal{Z}(\operatorname{Vec}_{B \times \widehat{B}}), \quad (a, \alpha) \mapsto (a, \alpha, \alpha, a),
\]
when $\operatorname{Vec}_{B \times \widehat{B}}$ is regarded as the Drinfeld center of $\operatorname{Vec}_B$. Thus, the associated braided $B$-crossed extensions have $B$-central extension structures $\operatorname{Vec}_B \to \operatorname{Vec}_{B \times \widehat{B}}$, $x \mapsto (x, 0)$, corresponding to:
\begin{align*}
   \iota &: \operatorname{Vec}_{B} \subset \mathcal{Z}(\operatorname{Vec}_{B \times \widehat{B}}), \quad a \mapsto (a, 0, 0, 0), \\
   \iota' &: \operatorname{Vec}_{B} \subset \mathcal{Z}(\operatorname{Vec}_{B \times \widehat{B}}), \quad a \mapsto (a, 0, 0, a).
\end{align*}

We want to see that those braided $G$-crossed extensions are equivalent (see Definition \ref{def:equivalence braided GX}). If we consider the identity functor of $\operatorname{Vec}_{B \times \widehat{B}}$ with tensor structure induced by the 2-cocycle $\langle -,- \rangle \in Z^2(B \times \widehat{B}, \mathbb{C}^*)$ given by $\langle (a,\alpha), (b,\beta) \rangle = \beta(a)$, then the induced braided autoequivalence of $\mathcal{Z}(\operatorname{Vec}_{B \times \widehat{B}})$ is
\begin{align*}
    \operatorname{Ind}(\langle -,- \rangle): \mathcal{Z}(\operatorname{Vec}_{B \times \widehat{B}}) \to \mathcal{Z}(\operatorname{Vec}_{B \times \widehat{B}}), \\
    (a, \alpha, \beta, b) \mapsto (a, \alpha, \beta - \alpha, a + b).
\end{align*}
Hence, $\operatorname{Ind}(\langle -,- \rangle) \circ \iota(a) = (a, 0, 0 - 0, a + 0) = (a, 0, 0, a) = \iota'(a)$. Therefore, the two braided $B$-crossed extensions are equivalent.

\subsection{The condensed fiber product}

If $\cC,\cD\in\cT(G,\cB_z)$ then $(\CGboxD)_e\supset \cB_z\boxtimes\cB_z$. Notice that $\cB_z\boxtimes\cB_z$ is a symmetric pointed braided fusion category with quadratic form $Q_z(\varphi,\psi):=\varphi(z)\psi(z)$.  Denote by $\nabla(\cB_z)$ the subcategory generated by objects of the form $(\varphi,\varphi^{-1})$.  Now $\nabla(\cB_z)\cong\Rep(B)$ is Tannakian since $Q_z(\varphi,\varphi^{-1})=\varphi(z)\varphi^{-1}(z)=1$.  Thus, we may de-equivariantize by $\nabla(\cB_z)$, see \cite[Definition 4.16]{DGNOI}.   This motivates:
\begin{definition}
    Let $\cC,\cD\in\cT(G,\cB_z)$.  The \textbf{condensed fiber product} $[\CGboxD]_B$ of $\cC$ and $\cD$ is the $B$-de-equivariantization of $\CGboxD$ via $\nabla(\cB_z)\cong \Rep(B)$.
\end{definition}
\begin{remark}
    In the case of $\cC,\cD\in\cT^{nd}(B,\cB_z)$ we have the following natural interpretation.  The Tannakian subcategory  $\nabla(\cB_z)\subset \cC\boxtimes\cD$ has centralizer precisely $\nabla(\cB_z)^\prime =\CGboxD$, so that the corresponding \emph{local} modules of $\Fun(B)$ in $\cC\boxtimes\cD$, \emph{i.e.}, the condensation of $\nabla(\cB_z)$, is exactly the condensed fiber product $[\cC\stackrel{B}{\boxtimes}\cD]_B$.
\end{remark}
It follows from results of \cite{DGNOI} and the above remark that if $\cC,\cD\in\cT^{nd}(B,\cB_z)$ then $[\cC\stackrel{B}{\boxtimes}\cD]_B$ is again non-degenerate.  We illustrate this construction with two simple examples.

\begin{example}\label{ex:su24z4}
    Consider the (non-degenerate) braided fusion category $\cC=SU(2)_4$.  This has objects $\unit,z,Y,X_1,X_{-1}$ where $\dim(X_i)=\sqrt{3}$, $\dim(Y)=2$, and $z$ is invertible. The objects $X_i$ are self-dual with \[X_i\otimes X_j=Y\oplus \delta_{i,j}\unit\oplus\delta_{i,-j}z.\] The subcategory $\langle z\rangle\cong \Rep(\Z/2)$ is symmetric and pointed and the  components of the induced $\Z/2$-grading have simple objects $\{\unit,z,Y\}$ (in $\cC_0$) and $\{X_1,X_{-1}\}$ (in $\cC_1$).  In particular, $\cC\in\cT(\Z/2,\Rep(\Z/2))$, where the $\Z/2$-action is trivial and the ($\Z/2$-crossed) braiding is the given braiding.  Another such category is $\Vec_{\Z/4}$, which we identify with $\Z/4$, written multiplicatively as $\langle g\rangle$.  The condensed fiber product $[\cC\stackrel{\Z/2}{\boxtimes}\Vec_{\Z/4}]_{\Z/2}$ is obtained from $\cC\stackrel{\Z/2}{\boxtimes}\Vec_{\Z/4}$ by condensing the subcategory generated by $(z,g^2)$.  The fiber product has $10$ objects, forming $5$ orbits under the $(z,g^2)\otimes -$ action, yielding $5$ simple objects, which we denote by $[X,1]$ in the trivial component and $[X,g]$ in the non-trivial component.  We can compute some fusion rules: \[[X_1,g]\otimes[X_1,g]=[X_1\otimes X_1,g^2]=[X_1\otimes X_1\otimes z,1]=[Y,1]\oplus[z,1].\]  In particular, the object $[X_1,g]$ is non-self-dual, with $[X_{-1},g]$ representing its dual.  Note the contrast with the category $\cC$, whose objects are all self-dual. Moreover, $[\cC\stackrel{\Z/2}{\boxtimes}\Vec_{\Z/4}]_{\Z/2}$ does not admit a braiding (see, e.g. \cite{BruillardOrtiz}).  We shall see below that it is still $\Z/2$-crossed braided, however.
\end{example}
\begin{example}\label{ex:z4}
   Consider the non-degenerate pointed category $\cP_4:=\cC(\Z/4,q)$ where $q(g^a):=e^{\pi ia^2/4}$ where we write $\Z/4$ multiplicatively with generator $g$.  The subcategory generated by $g^2$ is $\sVec$ since $q(g^2)=-1$, so that $\cP_4\in\cT^{nd}(\Z/2,\sVec)$. Thus we may construct the condensed fiber product $[\cP_4\stackrel{\Z/2}{\boxtimes}\cP_4]_{\Z/2}$ by condensing the diagonal subcategory $\langle (g^2,g^2)\rangle$ of $[\cP_4\stackrel{\Z/2}{\boxtimes}\cP_4]$.  We can calculate the result directly, or deduce it as follows: the resulting category is a non-degenerate pointed braided fusion category of dimension $4$, containing $\sVec$ with central charge $(e^{\pi i/4})^2=i$.  There is only one such category, namely $\Sem\boxtimes\Sem$.  
\end{example}

The following is one of our main results:

\begin{theorem}\label{thm:closed under cfp}
    The categories \( \cT(G,\cB_z) \) and \( \cT^{nd}(B,\cB_z) \) are closed under the operation of the condensed fiber product.
\end{theorem}

\begin{proof} 
    Let $A=\Fun(B)$ be the algebra object associated with the Tannakian category $\Rep(B)$, and define $\cF:=[\CGboxD]_B$.
Let $X$ and $Y$ be left $A$-modules with actions $\mu_X$ and $\mu_Y$ respectively, where $X\in \cF_e$ and $Y\in \cF_g$ as objects.  Recall from \cite{KiO} that the $A$-action $\mu_{X\otimes Y}$ on $X\otimes_A Y$ is given by either of $\mu_1:=\mu_X\otimes \Id_Y$ or $\mu_2:=(\Id_X\otimes \mu_Y)\circ (c_{A,X}\otimes \Id_Y)$.  Since $A\in (\cF_e)^\prime\cap \cF_e$ we have that $c_{A,X}c_{X,A}=\Id_{X\otimes A}$.  We must show that \begin{equation}\label{modulebraid}
    c_{X,Y}\circ(\Id_X\otimes \mu_Y)\circ (c_{A,X}\otimes \Id_Y)=(\mu_Y\otimes \Id_X)\circ (\Id_A\otimes c_{X,Y})
\end{equation} where we are using $\mu_{X\otimes Y}=\mu_2$ here.  We compose both sides of \eqref{modulebraid} with the isomorphism $c_{X,A}\otimes \Id_Y$ and compare as follows:
\[(c_{X,Y}\circ(\Id_X\otimes \mu_Y)\circ (c_{A,X}\otimes \Id_Y))\circ (c_{X,A}\otimes \Id_Y)=c_{X,Y}\circ(\Id_X\otimes \mu_Y)\]
while 
\[
    ((\mu_Y\otimes \Id_X)\circ (\Id_A\otimes c_{X,Y}))\circ (c_{X,A}\otimes \Id_Y)=(\mu_Y\otimes \Id_X)\circ (c_{X,A\otimes Y})=
c_{X,Y}\circ (\Id_X\otimes\mu_Y)
\]
where the last equality is due to functoriality of the braiding.  Thus we have verified \eqref{modulebraid}.  This implies that $(\cF_e)_A$ has a half-braiding with $(\cF_g)_A$ for all $g$; in particular we have a central functor $(\cF_e)_A\hookrightarrow \cF_A$.  Thus we have that $\cF_B$ is a $G$-crossed braided fusion category.

    Since the tensor product action of $\nabla(\cB_z)$ on $\cB_z\boxtimes\cB_z\subset [\CGboxD]_e$ is fixed-point-free, its image in $\cF_e$ may be identified with $\cB_z\boxtimes \Vec\cong\cB_z\subset\cF_e$, and remains transparent.  Thus $\cF\in\cT(G,\cB_z)$.

    Finally, we have already observed that if $\cC$ and $\cD$ are in $\cT^{nd}(B,\cB_z)$ then so is their condensed fiber product.  Thus $\cT^{nd}(B,\cB_z)$ is also closed under the condensed fiber product.
\end{proof}

Recall from \cite{DNO} the definition of fusion categories over a symmetric category \( \cE \). A fusion category over \( \cE \) is a fusion category \( \cC \) equipped with a braided functor \( T_\cC: \cE \to \mathcal{Z}(\cC) \), where \( \mathcal{Z}(\cC) \) denotes the Drinfeld center of \( \cC \). The primary condition is that the composition of \( T_\cC \) with the forgetful functor \( \mathcal{Z}(\cC) \to \cC \) must be faithful.

A 1-morphism \( F: \mathcal{A} \rightarrow \mathcal{B} \) between fusion categories over \( \cE \) is a tensor functor \( F: \mathcal{A} \rightarrow \mathcal{B} \) with a tensor isomorphism \( u_E: F(T_{\mathcal{A}}(E)) \rightarrow T_{\mathcal{B}}(E) \) for \( E \in \mathcal{E} \), satisfying the commutativity of the diagram:
\[
\begin{tikzcd}
F(T_{\mathcal{A}}(E) \otimes X) \arrow[r, "F_{T(E), X}"] \arrow[d] & F(T_{\mathcal{A}}(E)) \otimes F(X) \arrow[r, "u_E \otimes \text{id}"] & T_{\mathcal{B}}(E) \otimes F(X) \arrow[d] \\
F(X \otimes T_{\mathcal{A}}(E)) \arrow[r, "F_{X, T(E)}"] & F(X) \otimes F(T_{\mathcal{A}}(E)) \arrow[r, "\text{id} \otimes u_E"] & F(X) \otimes T_{\mathcal{B}}(E)
\end{tikzcd}
\]
where the vertical arrows represent the central structures of \( T_{\mathcal{A}}(E) \) and \( T_{\mathcal{B}}(E) \), respectively.

A 2-morphism is defined as a tensor natural transformation \( a: F \rightarrow G \) between tensor functors over \( \mathcal{E} \), satisfying the commutativity of:
\[
\begin{tikzcd}
F(T_{\mathcal{A}}(E)) \arrow[r, "a_{T(E)}"] \arrow[dr, "u_E"'] & G(T_{\mathcal{A}}(E)) \arrow[d, "v_E"] \\
& T_{\mathcal{B}}(E)
\end{tikzcd}
\]
for any \( E \in \mathcal{E} \). Here, \( u_E \) and \( v_E \) denote the structures of the functors \( F \) and \( G \) over \( \mathcal{E} \). For further details, see \cite[Section 2.5]{DNO}.

According to \cite{DNO}, fusion categories over \( \cE \) form a 2-category; for the definitions of monoidal functors and natural transformations over \( \cE \), we refer the reader to the same source. Moreover, a tensor product $\boxtimes_{\cE}$ is defined on this 2-category, making it into a monoidal 2-category.

The categories in \( \cT(G, \cB_z) \) and \( \cT^{nd}(G, \cB_z) \) can be viewed as objects in the 2-category of fusion categories over \( \cB_z \), since any \( G \)-crossed braided fusion category \( \cC \) is a central extension via its \( G \)-braiding. Specifically, the \( G \)-braiding induces a braided functor from the trivial component \( \cC_e \) to the Drinfeld center \( \mathcal{Z}(\cC) \). Since the composition with the forgetful functor \( \mathcal{Z}(\cC) \to \cC \) is simply the inclusion of \( \cC_e \) into \( \cC \), we obtain a braided functor \( T_\cC: \cB_z \to \mathcal{Z}(\cC) \).

Given fusion categories \( \cC \) and \( \cD \) over a symmetric category \( \cE \), their tensor product \( \cC \boxtimes_\cE \cD \) is constructed as follows. Consider the algebra \( A \) in \( \mathcal{Z}(\cC \boxtimes \cD) \) defined by
\[
A = (T_\cC \boxtimes T_\cD)(R(1)),
\]
where \( T_\cC \) and \( T_\cD \) are the braided functors from \( \cE \) to \( \mathcal{Z}(\cC) \) and \( \mathcal{Z}(\cD) \), respectively, and \( R: \cE \to \cE \boxtimes \cE \) is the right adjoint to the tensor product functor \( \otimes: \cE \boxtimes \cE \to \cE \). Then \( \cC \boxtimes_\cE \cD \) is realized as the category of (right) \( A \)-modules in \( \cC \boxtimes \cD \), denoted \( (\cC \boxtimes \cD)_A \). We refer to \( A \) as the \emph{diagonal algebra} in \( \mathcal{Z}(\cC \boxtimes \cD) \).

The following result relates the condensed fiber product to the tensor product of fusion categories over \( \cB_z \).

\begin{lemma}
Let \( \cC, \cD \in \cT(G, \cB_z) \). The tensor product \(\cC \boxtimes_{\cB_z} \cD\) is \( G \times G \)-graded, and the fusion subcategory supported over the diagonal \(\Delta(G) = \{(g, g) : g \in G\}\) is exactly the condensed fiber product.
\end{lemma}

\begin{proof}
Since \( \cB_z \) is pointed with simple objects \(\varphi \in \hB\), the adjoint to the tensor product is given by
\[
R(\varphi) = \bigoplus_{\tau, \gamma \in \hB : \tau \gamma = \varphi} \tau \boxtimes \gamma.
\]
In particular, for \( \varphi = \unit \), we have
\[
R(\unit) = \bigoplus_{\varphi \in \hB} \varphi \boxtimes \varphi^{-1} \cong \mathrm{Fun}(B).
\]
So the diagonal algebra is \( \mathrm{Fun}(B) \in \nabla(\cB_z) \subset \mathcal{Z}(\cC \boxtimes \cD) \), where \( \nabla(\cB_z) \) is the symmetric subcategory generated by objects \( (\varphi, \varphi^{-1}) \) with \( \varphi \in \hB \).

Since \( \cC \boxtimes \cD \) is canonically \( G \times G \)-graded and the diagonal algebra is in the \( (e,e) \)-component, \( \cC \boxtimes_{\cB_z} \cD \) is \( G \times G \)-graded. By definition, the condensed fiber product corresponds to the fusion category supported over \( \Delta(G) \).
\end{proof}

\begin{corollary}
The condensed fiber product induces a monoidal structure on \( \cT(G, \cB_z) \) and \( \cT^{nd}(G, \cB_z) \).
\end{corollary}

\begin{proof}
As shown in \cite[Section 2.5]{DNO}, the tensor product of fusion categories over a symmetric category defines a monoidal structure. For $\TGB$ and $\ndTBB$, the tensor product \( \boxtimes_{\cB_z} \) is not closed. To obtain a category in $\TGB$ (resp. $\ndTBB$), we restrict to the category supported over the diagonal subgroup \( \Delta(G) \) (resp. \(\Delta(B)\)). By Theorem~\ref{thm:closed under cfp}, this restriction is well-defined. Thus, the condensed fiber product defines a tensor product on \( \cT(G, \cB_z) \) and \( \cT^{nd}(B, \cB_z) \).
\end{proof}

\section{Relating Zesting and the Condensed Fiber Product}\label{sec:zesting_cfp}

We will show that in many cases, zesting can be achieved via the condensed fiber product. As a preliminary step, we recall the classification of $G$-crossed braided fusion categories using the 2-categorical Picard groups, as described in \cite{ENO3} and \cite{DN}.

Given a braided fusion category $\mathcal{B}$, we are interested in two types of $G$-extensions $\mathcal{B} \subset \mathcal{C} = \bigoplus_{g \in G} \mathcal{C}_g$. The first is the $G$-crossed braided extension (equivalently, a central extension), and the second is the $G$-braided extension (in the case where $G$ is abelian).

In \cite{DN}, $G$-crossed extensions of a braided category $\mathcal{B}$ are classified via the Picard 2-categorical group of invertible $\mathcal{B}$-module categories. Consequently, equivalence classes of such extensions can be described in terms of cohomology groups associated with a homomorphism $G \to \operatorname{Pic}(\mathcal{B})$. The zesting construction provides a direct interpretation of this classification in terms of the homomorphism $G \to \operatorname{Pic}(\mathcal{B})$, as discussed in \cite[Theorem 3.12.]{DGPRZ2}.  

\begin{theorem}\cite[Theorem 3.12]{DGPRZ2}\label{theo DGPRZ2}
Two $G$-crossed braided fusion extensions of a braided fusion category $\cB$ correspond to the same group homomorphism $G \to \operatorname{Pic}(\mathcal{B})$ if and only if they are related by $G$-crossed braided zesting.
\end{theorem}
\qed

\begin{corollary}
\label{lem: pointed is trivial homomorphism}
A $G$-crossed extension of a pointed braided fusion category $\cB$ is pointed if and only if the associated homomorphism $G \to \operatorname{Pic}(\cB)$ is trivial. Moreover, all pointed braided $G$-crossed extensions of $\cB$ are related by $G$-zesting.
\end{corollary}

\begin{proof}
In general, a $G$-crossed extension $\mathcal{C}$ of a braided category $\mathcal{B}$ contains an invertible object in $\mathcal{C}_g$ if and only if $\mathcal{C}_g \cong \mathcal{B}$ as $\mathcal{B}$-module categories \cite{GALINDO2011233}. In particular, a $G$-extension of a pointed category is pointed if and only if the associated homomorphism is trivial. Hence, by Theorem \ref{theo DGPRZ2}, any pair of pointed $G$-crossed braided extensions are related by $G$-zesting.
\end{proof}

\begin{remark}
Now suppose we are in the special case where $\mathcal{P} \in \ndTBB$ is pointed and $\cP_e=\cB_z$. In this case, $\mathcal{P}$ is a pointed minimal modular extension of $\mathcal{B}_z$, and we have already seen in Lemma \ref{lem:allNDpointedzest} that $\mathcal{P}$ must be a $B$-braided zesting of $\mathcal{Z}(\mathcal{B}_z)$. We can also obtain the explicit zesting directly. Indeed, the underlying $B$-graded fusion category has already been identified: it is a zesting of $\Vec_B \boxtimes \mathcal{B}_z$. Notice that, as a fusion category, $\mathcal{Z}(\mathcal{B}_z)$ is equivalent to $\Vec_B \boxtimes \mathcal{B}_z$. The braiding 
\[
Z_{gh} \alpha(g,h) = Z_g Z_h \cong Z_h Z_g = Z_{gh} \alpha(h,g)
\]
induces an isomorphism $t(g,h): \alpha(g,h) \cong \alpha(h,g)$. Analogous arguments, as above, using the hexagon equations on $\mathcal{P}$ imply that $(\alpha, \nu, t)$ is a braided $B$-zesting of $\mathcal{Z}(\mathcal{B}_z)$.
\end{remark}

\begin{theorem}\label{thm:cfp is zesting}
\begin{enumerate}
    \item Suppose $\mathcal{C}, \mathcal{P} \in \TGB$ with $\mathcal{P}$ pointed and $\mathcal{P}_e = \mathcal{B}_z$. Then, as a $G$-crossed braided fusion category, $(\mathcal{C} \Gbox \mathcal{P})_B$ is a zesting of $\mathcal{C}$.
    
    \item Suppose $\mathcal{C}, \mathcal{P} \in \ndTBB$ with $\mathcal{P}_e = \mathcal{B}_z$ as above. Then $(\mathcal{C} \Gbox \mathcal{P})_B$ is a braided $B$-zesting of $\mathcal{C}$.
\end{enumerate}
\end{theorem}

\begin{proof}
(1) We recall a realization of tensor products of $\mathcal{B}$-module categories, which is useful for this proof, as described in \cite{ENO3}. For $\mathcal{B}$-module categories $\mathcal{M}$ and $\mathcal{N}$, the tensor product $\mathcal{M} \boxtimes_{\mathcal{B}} \mathcal{N}$ is realized via the \emph{diagonal algebra} $A = \bigoplus_{X \in \operatorname{Irr}(\mathcal{B})} X \boxtimes X^* \in \mathcal{B} \boxtimes \mathcal{B}$, as $A$-modules in $\mathcal{M} \boxtimes \mathcal{N}$, seen as a $\mathcal{B} \boxtimes \mathcal{B}$-module category.

In the particular case where $\mathcal{B} = \mathcal{B}_z$ is a pointed symmetric fusion category, the \emph{diagonal algebra} is the algebra $Fun(B) \in \nabla(\mathcal{B}_z)$, previously considered in the definition of the condensed fiber product. Hence, for $\mathcal{C}, \mathcal{D} \in \cT(G, \mathcal{B}_z)$, the $g$-component of the condensed fiber product $[\mathcal{C} \Gbox \mathcal{D}]_B$ is exactly $\mathcal{C}_g \boxtimes_{\mathcal{B}_z} \mathcal{D}_g$.

By Corollary \ref{lem: pointed is trivial homomorphism}, if $\mathcal{D} = \mathcal{P}$ is pointed, then $\mathcal{P}_g \cong \mathcal{B}_z$, so 
\[
\left( [\mathcal{C} \Gbox \mathcal{P}]_B \right)_g = \mathcal{C}_g \boxtimes_{\mathcal{B}_z} \mathcal{P}_g = \mathcal{C}_g \boxtimes_{\mathcal{B}_z} \mathcal{B}_z \cong \mathcal{C}_g.
\]
This implies, in particular, that the group homomorphism $G \to \operatorname{Pic}(\mathcal{C}_e)$ does not change under the condensed fiber product for any pointed extension of $\mathcal{B}_z$. Hence, by Theorem \ref{theo DGPRZ2}, they are related by zesting.

(2) It follows directly from part (1) and \cite[Theorem 3.16]{DGPRZ2}, where it was shown that braided extensions structures over a zestings correspond precisely to braided zestings.
\end{proof}

We conclude with detailed examples illustrating the relationship between the condensed fiber product and zesting, as established by the previous results.

\begin{example}
    Recall that in Example \ref{ex:su24z4} we constructed the condensed fiber product $[SU(2)_4\stackrel{\Z/2}{\boxtimes}\Vec_{\Z/4}]_{\Z/2}$ and discovered that it is not braided, although it is $\Z/2$-crossed braided.  We can easily see that $\Vec_{\Z/4}$ may be obtained by $\Z/2$-zesting $\Vec_{\Z/2}\boxtimes\Rep(\Z/2)$ using the $2$-cocycle defined by $\lambda(1,1)=(e,\varphi)$ where $\varphi$ is the generator of $\Rep(\Z/2)$, with trivial $3$-cochain $\nu=1$.  Similarly, if we take the same $\Z/2$-crossed zesting datum $(\lambda,1)$ and apply it to $SU(2)_4$ then we obtain the same fusion category as in Example \ref{ex:su24z4}.  Thus we have:
    \[[SU(2)_4\stackrel{\Z/2}{\boxtimes}(\Vec_{\Z/2}\boxtimes\Rep(\Z/2))^{(\lambda,1)}]_{\Z/2}\cong (SU(2)_4)^{(\lambda,1)} \] \emph{as $\Z/2$-graded fusion categories} i.e., the condensed fiber product of $SU(2)_4$ with $\Vec_{\Z/4}$ can be obtained by $\Z/2$-zesting $SU(2)_4$ with the same zesting datum used to obtain $\Vec_{\Z/4}$ via a  $\Z/2$-zesting of the canonical pointed category $\Vec_{\Z/2}\boxtimes\Rep(\Z/2)\in\cT(\Z/2,\Rep(\Z/2))$.
\end{example}

\begin{example}
    In Example \ref{ex:z4}, it was shown that $[\cC(\Z/4,q)\stackrel{\Z/2}{\boxtimes}\cC(\Z/4,q)]_{\Z/2}$, with $q(g^a)=e^{\pi i a^2/4}$, is equivalent to $\Sem^{\boxtimes 2}$.  We may apply the results of \cite[Propositions 6.4,6.5]{DGPRZ} to verify that there is a braided ribbon $\Z/2$-zesting datum $(\lambda,\nu,t,f)$ such that $\cZ(\Z/2)^{(\lambda,\nu,t,f)}\cong \cC(\Z/4,q)$, and that, moreover, $\cC(\Z/4,q)^{(\lambda,\nu,t,f)}\cong \Sem^{\boxtimes 2}$.  Since all of these categories have trivial component $\sVec=\langle f\rangle$, we find that $\lambda$ is the unique choice such that $\lambda(1,1)=f$, the fermion.  Notice that this changes $\Z/4$ fusion rules into $\Z/2\times \Z/2$ fusion rules, and vice versa.  A standard choice for $\nu$ is given in \cite[Proposition 6.4]{DGPRZ} (taking $b=0$ in the notation there) and then the choices of $t$ correspond to a root of unity $s$ with $s^2=-i$, for which we may take the choice $s=e^{-\pi i/4}$.  Ultimately, we find that the (unique unitary) ribbon zesting has the effect of multiplying the twists of the objects in the non-trivial component by $s^{-1}=e^{\pi i/4}$.  Thus the twists in $\cZ(\Z/2)$ change to those of $\cC(\Z/4,q)$ as follows: \[[1,-1,1,1]\stackrel{(\lambda,\nu,t,f)}{\rightarrow}[1,-1,e^{\pi i/4},e^{\pi i/4}],\] which, in turn change to those of $\Sem^{\boxtimes 2}$ as follows: \[[1,-1,e^{\pi i/4},e^{\pi i/4}]\stackrel{(\lambda,\nu,t,f)}{\rightarrow}[1,-1,i,i].\]  Thus we have shown that, in this case, the condensed fiber product of $\cC$ with $\cZ(\cB_z)^{(\lambda,\nu,t)}$ is equivalent to $\cC^{(\lambda,\nu,t)}$.
\end{example}  

\begin{remark}
Computing the multiplicative central charge $\sum_i d_i^2 \theta_i / \dim(\cC)$ directly from ribbon zesting data can be tedious. However, computing the central charge of a condensed fiber product is straightforward. 

Indeed, let $\cC, \cP \in \ndTBB$ with $\cP_e = \cB_z$ and central charges $\chi(\cC)$ and $\chi(\cP)$, respectively. Then the condensed fiber product $[\cC \stackrel{B}{\boxtimes} \cP]_B$ has a multiplicative central charge $\chi(\cC) \cdot \chi(\cP)$. This follows directly from the definition. Since the central charge of pointed categories is relatively simple to compute, identifying a zesting with a condensed fiber product as described above allows us to compute the central charge of a zesting more easily. In particular, since the central charge of a pointed modular category is an 8th root of unity, zesting cannot significantly alter the central charge.
\end{remark}

\section{Discussion}\label{sec:discussion}

Theorem \ref{thm:cfp is zesting} states that the condensed fiber product with pointed (braided) extensions corresponds to (braided) zestings. Hence, the condensed fiber product can be seen as a generalization of zesting. However, zesting has advantages from a computational perspective, as it allows for explicit formulas for fusion rules, $F$-matrices, $R$-matrices, and modular data, as established in \cite{DGPRZ}.

The investigation of the condensed fiber product and its relationship with zesting also suggests several further avenues of exploration.  We list a few that will be left to future work.
\begin{enumerate}
    \item The condensed fiber product (and zesting) can be applied to more general settings: take any two $G$-graded fusion categories $\cC$ and $\cD$ such that $\cB_z\subset\cC_e$ and $\cB_z\subset\cD_e$, but is not necessarily transparent in the respective trivial components.  Then one may still form the condensed fiber product, since $\nabla(\cB_z)$ remains an algebra object in $\cC\boxtimes\cD$.  Moreover, the resulting category is $G$-graded, since $\nabla(\cB_z)$ lies in the trivially graded component of $\cC\boxtimes\cD$.  
    \item In the case of two non-degenerate categories $\cC,\cD$ containing a common symmetric pointed subcategory $\cB_z$, we have an interpretation of the condensed fiber product as the local $\Fun(B)$-modules in $\cC\boxtimes\cD$, i.e., $[\cC\boxtimes\cD]_B=(\cC\boxtimes\cD)_{\Fun(B)}^{loc}$.  Suppose $\cC$ and $\cD$ are non-degenerate fusion categories with a common condensible algebra $\cA$. 
 Then, as in \cite{DNO}, one has a diagonal condensible algebra object $\Hc$ in $\cA\boxtimes\cA\subset \cC\boxtimes\cD$ so that $(\cC\boxtimes\cD)_\Hc^{loc}$ is again a non-degenerate category containing $\cA$.  For example, if one has $\Rep(G)\subset \cC$ for some non-abelian group $G$ one may take $\cD$ to be any of the zestings of $\cZ(\Rep(G))$, for example $\Rep(D^\omega G)$ for some $3$-cocycle $\omega$.  Then this more general condensed fiber product could be thought of as a generalized zesting of $\cC$.  
\end{enumerate}

\bibliography{refs} 
\bibliographystyle{plain}

\end{document}